\documentclass[11pt]{amsart}

\usepackage[a4paper,margin=1in]{geometry}
\usepackage{amsmath,amssymb,amsthm,mathtools}
\usepackage{mathrsfs}
\usepackage{enumitem}
\usepackage{microtype}
\usepackage[colorlinks=true,linkcolor=blue,citecolor=blue,urlcolor=blue]{hyperref}
\usepackage{cleveref}
\usepackage{comment}
\usepackage{graphicx}
\usepackage{todonotes}

\newtheorem{theorem}{Theorem}[section]
\newtheorem{lemma}[theorem]{Lemma}
\newtheorem{proposition}[theorem]{Proposition}

\usepackage{lean-code}

\usepackage{biblatex}
\newcommand{\C}{\mathbb C}
\newcommand{\vc}{\mathrm{vc}}
\newcommand{\R}{\mathbb R}
\newcommand{\N}{\mathbb N}
\newcommand{\F}{\mathcal F}
\newcommand{\dd}{\,d}

\newcommand{\Polyn}[2]{\mathrm{Poly}_{#1}(#2)}

\DeclareMathOperator{\spanop}{span}

\title{Phase retrieval from a uniformly discrete point set}

\author[Jaume de Dios Pont]{Jaume de Dios Pont}
\address{Center for Data Science, New York University, New York, New York 10011, USA}
\email{jdedios@nyu.edu}

\author[Josef Greilhuber]{Josef Greilhuber}
\address{Department of Mathematics, MIT, Cambridge, MA 02139, USA}
\email{jgreil@mit.edu}

\author[Lukas Liehr]{Lukas Liehr}
\address{Department of Mathematics, Bar-Ilan University, Ramat-Gan 5290002, Israel}
\email{lukas.liehr@biu.ac.il}

\author[Mitchell A. Taylor]{Mitchell A. Taylor}
\address{Mathematical Institute, University of Oxford, Andrew Wiles Building, Radcliffe Observatory Quarter, Woodstock Road, Oxford, OX2 6GG, United Kingdom}
\email{mitchtaylor@shaw.ca}

\date{\today}

\begin{document}

\begin{abstract}
We prove that for every window $w$ of the form $w(x) = e^{-\pi |x|^2} h(x)$, where $h$ is a polynomial, there exists a uniformly discrete set of points $\mathcal{S} \subset \mathbb R^{2d}$ such that the magnitude of the short-time Fourier transform $V_w f$ on $\mathcal{S}$ determines every $f \in L^2(\mathbb R^d)$ up to a constant phase factor. The separation distance can be chosen independent of the degree of $h$ and proportional to the square root of the dimension. The proof combines a discrete norming inequality, based on a multidimensional Remez inequality and VC-dimension bounds, with tail estimates for the reproducing kernel. A formalization of our main result in Lean 4 is also provided.
\end{abstract}


\maketitle

\section{Introduction}

\label{sec:introduction}

\subsection{}

Let $X$ be a space of functions defined on a domain $\Omega$. 
The phase retrieval problem in $X$ asks to recover $F \in X$, up to a constant phase factor, from the absolute values $(|F(x)|)_{x \in \Omega}$. 
The problem originally arises as a signal recovery question from experimental measurements, including in crystallography and in coherent diffraction imaging techniques such as ptychography. 
These are problems where one wishes to recover properties of a physical system by measuring the light wave diffraction patterns it generates. 
In this case, one can typically measure the amplitude of the diffracted light wave (its square, $|F(x)|^2$, is proportional to the energy deposited on a detector) but one cannot measure the phase of the wave. 
Due to the nature of the diffraction process, the geometric properties of the object that one wishes to measure are deeply intertwined in both the phase and the amplitude of the diffracted wave, and one must therefore recover the phase as well.  

A natural question is whether one can retain uniqueness after restricting the measurements to a uniformly discrete subset of $\Omega$. In other words, one wishes to recover $F$ (up to a global phase), from the values $(|F(x)|)_{x\in \mathcal S}$,  for some set $\mathcal S\subset \Omega$ such that any two pairwise distinct elements $s,s' \in \mathcal{S}$ satisfy $|s-s'| \geq \delta$. Such discretization questions are motivated by experimental limitations in signal processing and imaging, where measurements often provide only sampled magnitude data \cite{GrohsKoppensteinerRathmair2020} at a resolution limited by the measuring device. Questions of this nature have been studied in several settings, including band limited functions \cite{alaifari2021uniqueness,alaifari2017reconstructing,Thakur2011}, shift-invariant spaces \cite{ChenChengSunWang2020,grochenig2020phase,romero2021sign}, and the short-time Fourier transform (STFT) \cite{alaifari2025multiwindow,chen2025sampling,GrohsLiehr2022,GrohsLiehr2025}.

This work considers the case where $X$ is the range of $L^2(\mathbb R^d)$ by the short-time Fourier transform (STFT) with a window function $e^{-\pi |x|^2} h(x)$ where $h$ is a polynomial. From a physical point of view, this is intimately related to the recovery question in ptychography. Following a classical argument, the Bargmann transform identifies the range of the STFT with a certain space of polyanalytic functions which depends on the chosen window function. The bulk of the work of the present paper follows after this reduction: We prove that every such polyanalytic Fock space admits a uniformly discrete phase-retrieval set, yielding a uniformly discrete phase-retrieval set for the STFT in every dimension and for a large and natural class of window functions.

For a true polyanalytic Fock space on $\C^d$, uniqueness is known to hold when $\mathcal S = \C^d$ \cite{AlaifariWellershoff2021}. The remaining challenge, showing that a uniformly discrete set of observables exists, has raised significant attention over the past years. A main obstruction, already in dimension one, is that no lattice allows for phase retrieval of every function in $L^2(\R)$. This was originally shown for the Gaussian window \cite{AlaifariWellershoff2022} and later for general functions \cite{GrohsLiehr2022}. On the other hand, sufficiently dense square-root lattices give uniqueness for a class of windows that includes the windows under consideration here \cite{GrohsLiehr2025}. Such sets are discrete, but the distances between distinct sampling points becomes arbitrarily small, making measurements impractical. Moreover, it was shown in \cite{grohs2025phase} that a suitable union of three uniformly discrete sets forms a phase retrieval set, although the resulting union is not itself uniformly discrete. Uniformly discrete phase retrieval sets are also known to exist for certain dense subsets of $L^2(\R^d)$. However, whether such a set exists for the entire space $L^2(\R^d)$ has remained an open problem.

Another reason for the difficulty is that the polyanalytic Fock spaces that arise from the Bargmann transform (which are Hilbert spaces where all functions are integrable against $e^{-\pi|z|^2}$) are \emph{critical} for this problem: If one improves the weight to $e^{-\pi|z|^{2-\epsilon}}$ there are discrete STFT phaseless uniqueness sets of arbitrarily high separation. If one instead relaxes the integrability condition to $e^{-\pi|z|^{2+\epsilon}}$, there are no uniformly discrete sets of Fourier uniqueness.

\subsection{}

The short-time Fourier transform (STFT) of $f \in L^2(\R^d)$ with respect to the window function $h \in L^2(\R^d)$ is defined by
\begin{equation}\label{eq:intro_stft}
    V_h f(x,\omega) = \int_{\R^d} f(t)\overline{h(t-x)}e^{-2\pi i\omega\cdot t}\, dt,
    \quad x,\omega \in \R^d.
\end{equation}
The STFT phase retrieval problem consists in determining $f$ from the collection of magnitude measurements $\{ |V_h f(x,\omega)| : (x,\omega) \in \mathcal{S} \}$, where $\mathcal{S} \subseteq \R^{2d}$. Since $|V_h (\sigma f)|=|V_hf|$ for any $\sigma \in \mathbb C$ with $|\sigma|=1$, recovery is only possible up to a global phase factor. We restrict ourselves to windows of the form $$w(x)= e^{-\pi |x|^2} h(x)$$ where $h(x) \in \mathbb C[x_1,\ldots,x_d]$ is an arbitrary polynomial. 
It is well-known that for these windows, the magnitude of the STFT on all of $\R^{2d}$ determines every $f\in L^2(\R^d)$ up to a global phase. Indeed, the ambiguity-function identity gives this conclusion for every window $h$ with the property that $V_h h$ is nonzero almost everywhere \cite{AlaifariWellershoff2022}. For the windows under consideration, this condition follows from the fact that its ambiguity function is a Gaussian multiplied by a nonzero polynomial.

\subsection{} We say that a subset $\mathcal S$ of a domain $\Omega$ is a \emph{phase-retrieval set} for a space of functions $X$ on $\Omega$ if, for every $F,G\in X$, we have the implication
$$
    |F(s)|=|G(s)|\quad \text{for all }s\in\mathcal S
    \quad\Longrightarrow\quad
    F=\sigma G\quad \text{for some } \sigma\in\C,\ |\sigma|=1.
$$
A subset $\mathcal S$ of $\mathbb R^{m}$ is called \emph{$\delta$-separated} if $|s-t|\geq\delta$ for all distinct $s,t\in\mathcal S$, and \emph{uniformly discrete} if this holds for some $\delta>0$. Our main result shows that there exist uniformly discrete phase-retrieval sets for the subspaces of $L^2(\mathbb R^d \times \mathbb R^d)$ given by the range of the STFT with respect to the windows introduced above.

\begin{theorem}[STFT-version]\label{thm:STFT}
For every \(d\ge 1\) and every window $w(x) = e^{-\pi |x|^2} h(x)$ with $h \in \mathbb C[x_1,\ldots,x_d]$, $h \neq 0$, there exists a $\delta$-separated set \(\mathcal S \subset \R^d \times \R^d\) such that whenever $f,g \in L^2(\mathbb R^d)$ satisfy
\begin{align*}
    |V_w f (x,\omega)| = |V_w g (x,\omega)| \text{ for all } (x,\omega) \in \mathcal S,
\end{align*}
then there exists $\theta \in \mathbb C$ with $|\theta| = 1$ such that $f = \theta g$ almost everywhere. The separation distance $\delta$ can be taken to be at least $\delta = 10^{-7} \sqrt{d}$.
\end{theorem}

\subsection{}
The previous theorem will follow directly from a restated version of it which we provide in Section \ref{sec:rest}. To state our result, we briefly introduce the Fock spaces under consideration.

The (analytic) Fock space $\F^d_{\mathbf 0}$ is defined by
$$
    \F^d_{\mathbf 0}
    = \left\{F:\C^d\to\C \text{ entire}:
    \int_{\C^d}|F(z)|^2e^{-|z|^2}\dd z<\infty\right\}.
$$
Given $L \in \mathbb N$, the polyanalytic Fock space of order $L$ is
\begin{equation}\label{eq:intro_true_fock}
    \F^d_{L}
    = \left\{ F \in L^2(\mathbb C^d; e^{-|\cdot|^2} ): \partial_{\bar z}^{L+1} F = 0 \right\}.
\end{equation}
Here, $\partial_{\bar z} F = (\partial_{\bar z_1} F, \ldots, \partial_{\bar z_d} F)$ denotes the collection of Wirtinger derivatives in $\bar z_1,\ldots,\bar z_d$. This agrees with the definition as a closed span of complex Hermite polynomials given in Section~\ref{sec:tails}. In particular, $L = 0$ gives the analytic case.

Let us restrict ourselves momentarily to the tensor-product Hermite windows $$h_{\mathbf q}(t)=\prod_{j=1}^d h_{q_j}(t_j),$$ where $h_n$ denotes the normalized Hermite function of degree $n$, with $h_0(t)=2^{1/4}e^{-\pi t^2}$.
The connection with the STFT is then given by the \emph{true polyanalytic Bargmann transform}, which is a unitary map $\mathcal B_{\mathbf q}:L^2(\R^d)\to\F^d_{\mathbf q}$ \cite[Section 3]{Abreu2010}. With $z=\sqrt\pi(x-i\omega)$, it can be written in our normalization as
\begin{equation*}
    (\mathcal B_{\mathbf q}f)(z)
    = e^{|z|^2/2+i\pi x\cdot\omega}V_{h_{\mathbf q}}f(x,\omega).
\end{equation*}
Consequently, we have that
$
    |V_{h_{\mathbf q}}f(x,\omega)|
    = e^{-|z|^2/2}|(\mathcal B_{\mathbf q}f)(z)|.
$
Since the Gaussian factor is nowhere zero, equality of STFT magnitudes on a set in $\R^{2d}$ is equivalent to equality of the corresponding Fock-space magnitudes on its image in $\C^d$. The change of variables also preserves uniform separation, up to the factor $\sqrt\pi$.

To deal with a window of type $w(x) = e^{-\pi |x|^2} h(x)$, $h(x) = \sum_{|\mathbf q| \leq L} a_{\mathbf q} x^{\mathbf q}$, we consider the generalized Bargmann transform
\begin{align}
    \label{eq:intro_bargmann}
    (\mathcal B_{h}f)(z)
        = e^{|z|^2/2+i\pi x\cdot\omega}V_{w}f(x,\omega),
\end{align}
and the space $\mathcal F_h = \{\mathcal B_{h}f: f \in L^2(\mathbb R^d)\}$. Since $h$ can be expanded into Hermite polynomials of degree at most $L$, it follows by linearity that $\mathcal F_h \subseteq \bigoplus_{|\mathbf q| \leq L} \mathcal F_{\mathbf q}^d$, so $ \mathcal F_h $ is a space of polyanalytic functions of order $L$. Further properties of this space, including a convenient orthonormal basis, are discussed in \Cref{sec:tails}.

\subsection{}\label{sec:rest} \Cref{thm:main}, restated in the language of the polyanalytic Fock spaces $\mathcal F_h$, reads as follows.

\setcounter{theorem}{0}
\begin{theorem}[Polyanalytic-version]\label{thm:main}
For every \(d\ge 1\) and every $h(x) = \sum_{|\mathbf q| \leq L} a_{\mathbf q} x^{\mathbf q}$, there exists a uniformly $\delta$-separated set \(\mathcal S \subset\C^d\) such that \(\mathcal S\) is a phase-retrieval set for the space \(\F_{h}\). The separation distance $\delta$ can be taken to be at least $\delta = \frac{4}{10^7} \sqrt{d}$.
\end{theorem}

Through \eqref{eq:intro_bargmann}, \Cref{thm:main} gives, for every $w(x) = e^{-\pi |x|^2} h(x)$, a uniformly discrete set $\mathcal S\subseteq\R^{2d}$ such that the values $|V_{w}f|$ on $\mathcal S$ determine every $f\in L^2(\R^d)$ up to a global phase.
Remarkably, the separation distance in this theorem does not depend on the degree of $h$, although the proof presented here makes substantial use of the fact that $h$ is a polynomial. With some additional effort, it can even be shown that there exists a $\delta$-separated set $\mathcal S \subseteq \mathbb C^d$ which is a phase-retrieval set for all spaces $\mathcal F_h$ \emph{simultaneously}. This raises the intriguing question whether this result in fact holds for a much larger class of windows.

We also remark that the dimension scaling seems to suggest that phase-retrieval sets become less dense as the dimension increases. This is not the case: A lattice of co-volume (density) 1 in $d$-dimensions can have separation on the order of $\sqrt{d}$ between its elements.

\subsection{}
The proof of \Cref{thm:main} combines a finite-dimensional sampling argument with an approximation estimate. In Section~\ref{sec:norming}, we use a multidimensional Remez inequality and a VC-dimension argument to construct separated sets that control differences $|F|^2-|G|^2$ for polynomials that are drawn from a given, low-dimensional subspace of the space of polynomials. Section~\ref{sec:tails} bounds the error incurred by truncating the complex Hermite expansion of a Fock-space function. In Section~\ref{sec:main_proof}, we apply these estimates on annuli whose radii increase sufficiently rapidly. On each annulus, we consider the truncated Taylor series $F_N$ and $G_N$ up to an order $N$ approximately the square radius of the annulus, use the norming inequality of Section~\ref{sec:norming} to construct a uniformly separated set $\mathcal S_N$ capable of retrieving the norm of $|F_N|^2-|G_N|^2$ up to an exponential loss, and show that the approximation error tends to zero faster than the exponential loss incurred from the norming inequality. It follows that $|F|=|G|$ everywhere, and continuous phase retrieval completes the proof.

\section{Discrete norming inequality}\label{sec:norming}

In this section, we prove the discrete norming inequality which provides the crucial propagation of smallness of $|F|^2-|G|^2$ from a discrete subset to the bulk of $\mathbb C^d$. For its formulation, the complex-analytic setting of the remainder of this paper is irrelevant. The polynomial mapping $(F,G) \mapsto |F|^2-|G|^2$ could be replaced by any other polynomial of low degree at little cost, but for the sake of obtaining explicit constants we chose to remain at the present level of generality.

\begin{proposition}[Discrete norming inequality]
\label{prop:norming}
Let $\Omega \subseteq \mathbb R^m$ be a bounded domain. Denote its convex hull by $\hat\Omega$. There exist constants \(\mu,\delta > 0\) such that for any vector subspace $\mathscr V \subseteq \Polyn{n}{\mathbb R^m}$ of complex dimension $N$ ($N,n \in \mathbb N$) there exists a $\delta N^{-\frac1{m}}$-separated set of points $\mathcal S_{\mathscr V} \subseteq \Omega$ with
\[
    \left\lVert |F|^2-|G|^2 \right\rVert_{L^\infty(\hat \Omega)}
    \le
    e^{\, \mu n}
    \left\lVert |F|^2-|G|^2 \right\rVert_{L^\infty(\mathcal S_{\mathscr V})}
\]
for all $F, G \in \mathscr V$. If $|\hat \Omega \setminus \Omega| \leq 2^{-2m-2}|\hat \Omega|$, we may choose $\delta = \frac1{16} \left( \frac{|\Omega|/|\mathbb B^{m}|}{3\cdot 10^5 (1 + d)} \right)^{\frac1{m}}$ and $\mu = 4$.
\end{proposition}

A crucial ingredient in the proof of \Cref{prop:norming} is that the collection of possible sublevel sets of the function $|F|^2-|G|^2$ is of bounded complexity and hence can be sampled comparably well. This is made precise in the following statement.

\begin{lemma}[VC-dimension bound]
\label{lem:VCdimension}
Let $\Omega$ and $\mathscr V$ be as in \Cref{prop:norming}. Then the family $\mathcal F$ of subsets of $\Omega$ given by
\[
    \mathcal F := \left\{ \left\{x \in \Omega: \big||F(x)|^2 - |G(x)|^2\big| \geq \tau \right\}, F,G \in \mathscr V, \tau > 0 \right\}
\]
has VC-dimension no larger than $36 N$.
\end{lemma}

\begin{proof}
    It suffices to bound the VC dimension of the family $\mathcal G_N$ of subsets of $\C^N$ given by
    \begin{align*}
        \mathcal G_N = \left\{ \left\{y \in \C^N: \left|\left|\textstyle\sum_{j=1}^{N} a_j y_j\right|^2 - \left|\textstyle\sum_{j=1}^{N} b_j y_j\right|^2\right| \geq 1 \right\}, a,b \in \mathbb C^N \right\}.
    \end{align*}
    To see this, pick a basis $(\phi_{j})_{j=1}^N$ of $\mathscr V$, consider the map $\Phi: \Omega \to \mathbb C^{N}$ given by 
    \begin{align*}
        \Phi(z) = \left( \phi_{j}(x)\right)_{j=1}^N,
    \end{align*}
    and observe that $\mathcal F = \{\Phi^{-1}(X): X \in \mathcal G_N\}$. VC-dimension is non-increasing under pull-back along arbitrary maps, hence $\vc\, \mathcal F \leq \vc\, \mathcal G_{N}$.

    It remains to estimate the VC dimension of $\mathcal G_N$. Suppose $K$ points $y^{1},\ldots,y^{K}$ are shattered by $\mathcal G_N$. Consider the collection of polynomials $p_{1},\ldots,p_{K}$ given by
    \begin{align*}
        p_{\ell}(a,b) = \left|\left|\textstyle\sum_{j=1}^{N} a_j y_j^{(\ell)}\right|^2 - \left|\textstyle\sum_{j=1}^{N} b_j y_j^{(\ell)} \right|^2\right|^2 - 1.
    \end{align*}
    These are $K$ polynomials of degree $4$ in $4N$ real variables, namely $\Re a_j, \Im a_j, \Re b_j, \Im b_j$ for $j=1,\ldots,N$. 
    For each subset $S$ of $\{1,\ldots, K\}$, there exist $a,b$ such that $p_{\ell}(a,b) \geq 0$ if and only if $\ell \in S$. After scaling $a$ and $b$ by a sufficiently small factor $\lambda > 1$, we have $p_{\ell}(a,b) > 0$ if $\ell \in S$ and $p_{\ell}(a,b) < 0$ if $\ell \not \in S$.
    
    Given a collection of polynomials $q_1,\ldots,q_m$ in $r$ real variables, a \emph{sign pattern} is a sequence $\sigma \in \{-1,1\}^m$ such that there exists $x \in \mathbb R^r$ with $\mathrm{sign}(q_j(x)) = \sigma_j$. Assuming the degree of $q_1,\ldots,q_m$ is bounded by $d \in \mathbb N$, Warren showed \cite[Theorem 3]{Warren1968} that the number of possible sign patterns does not exceed 
    \begin{align*}
        \left( \frac{4e d m}{r} \right)^r,
    \end{align*}
    where $e$ is Euler's number.

 Applied to the polynomials $p_{1},\ldots,p_{K}$ above, we find that 
    \begin{align*}
        2^K \leq \left( \frac{16e K}{4N} \right)^{4N},
    \end{align*}
    since all possible sign patterns occur. If $K \geq 4N$, taking logarithms and dividing by $4N$ yields
    \begin{align*}
        \frac K {4N} - \log_2 \frac K {4N} \leq 4 + \frac 1{\log 2}.
    \end{align*}
	Monotonicity of the function $x \mapsto x - \log_2 x$ on $(\frac1{\log 2},\infty)$ and $9 - \log_2 9 > 4 + \frac1{\log 2}$ imply that $K \leq 36N$. It follows that $\vc\, \mathcal G_{N} \leq 36 N$.
\end{proof}

\begin{proof}[Proof of \Cref{prop:norming}]
    A version of Brudnyi--Ganzburg's multivariate Remez inequality \cite[Eq. 4.4]{Ganzburg2001} states that for any convex domain $\Omega \subseteq \mathbb R^m$ and any subset $\omega \subseteq \Omega$ with $|\Omega \setminus \omega| \leq 2^{-2m}|\Omega|$, the bound
    \begin{align*}
        \| p \|_{L^\infty(\Omega)} \leq \exp \left(4 k \left( \frac{|\Omega \setminus \omega|}{|\Omega|} \right)^{\frac1{2m}} \right) \| p \|_{L^\infty(\omega)}
    \end{align*}
    holds for all real polynomials $p$ in $m$ variables of degree at most $k$. It follows that for any $F,G \in \mathscr V$, the Lebesgue measure of the superlevel set
    \begin{align*}
        \left\{x \in \Omega: \big||F(x)|^2-|G(x)|^2\big| \geq \exp(-8 n t^{\frac1{2m}}) \| |F|^2-|G|^2 \|_{L^\infty(\Omega)} \right\}
    \end{align*}
    is at least $t |\Omega|$. If $\Omega$ is not convex, but $t|\Omega|+|\hat \Omega \setminus \Omega| < 2^{-2m} |\hat \Omega|$,  the same argument applied to $\hat \Omega$ yields a lower bound of $t|\Omega|$ for the size of the superlevel sets 
    \begin{align*}
        \left\{x \in \Omega: ||F(x)|^2-|G(x)|^2| \geq \exp\left(- 8n \left( \frac{t|\Omega|+|\hat \Omega \setminus \Omega|}{|\hat \Omega|} \right)^{\frac1{2m}}\right) \| |F|^2-|G|^2 \|_{L^\infty(\hat \Omega)} \right\}
    \end{align*}
    instead. Let $\mathcal F$ denote the collection of such subsets of $\Omega$. By \Cref{lem:VCdimension}, the VC dimension of $\mathcal F$ is bounded by $36 N$.

    Now, we will pick a uniformly discrete subset $\mathcal S_{\mathscr V}$ of $\Omega$ with the property that each set in $\mathcal F$ contains at least one point of $\mathcal S_{\mathscr V}$, and hence,
    \begin{align*}
         \| |F|^2 - |G|^2 \|_{L^\infty(\hat \Omega)} \leq \exp\left(8n \left( \frac{t|\Omega|+|\hat \Omega \setminus \Omega|}{|\hat \Omega|} \right)^{\frac1{2m}}\right) \| |F|^2-|G|^2 \|_{L^\infty(\mathcal S_{\mathscr V})}
    \end{align*}
    for all $F,G \in \mathscr V$.
    We will do so by first picking $M$ points $x_1,\ldots,x_M$ in $\Omega$ uniformly at random, with $M$ a fixed multiple of $N$ to be determined later, then using Vapnik--Chervonenkis' relative inequality to show that with high probability, each set in $\mathcal F$ contains at least $\frac t2 M$ points, and finally verifying that we may discard $\frac t3 M$ points to ensure the remaining points in $\mathcal S_N$ are uniformly $\delta N^{-\frac1{m}}$-separated, for some $\delta > 0$ which depends on $\Omega$ only. \Cref{fig:norming} illustrates this process.

    \begin{figure}
        \centering
        \raisebox{0pt}[0.4\linewidth][0pt]{\includegraphics[width=0.4\linewidth]{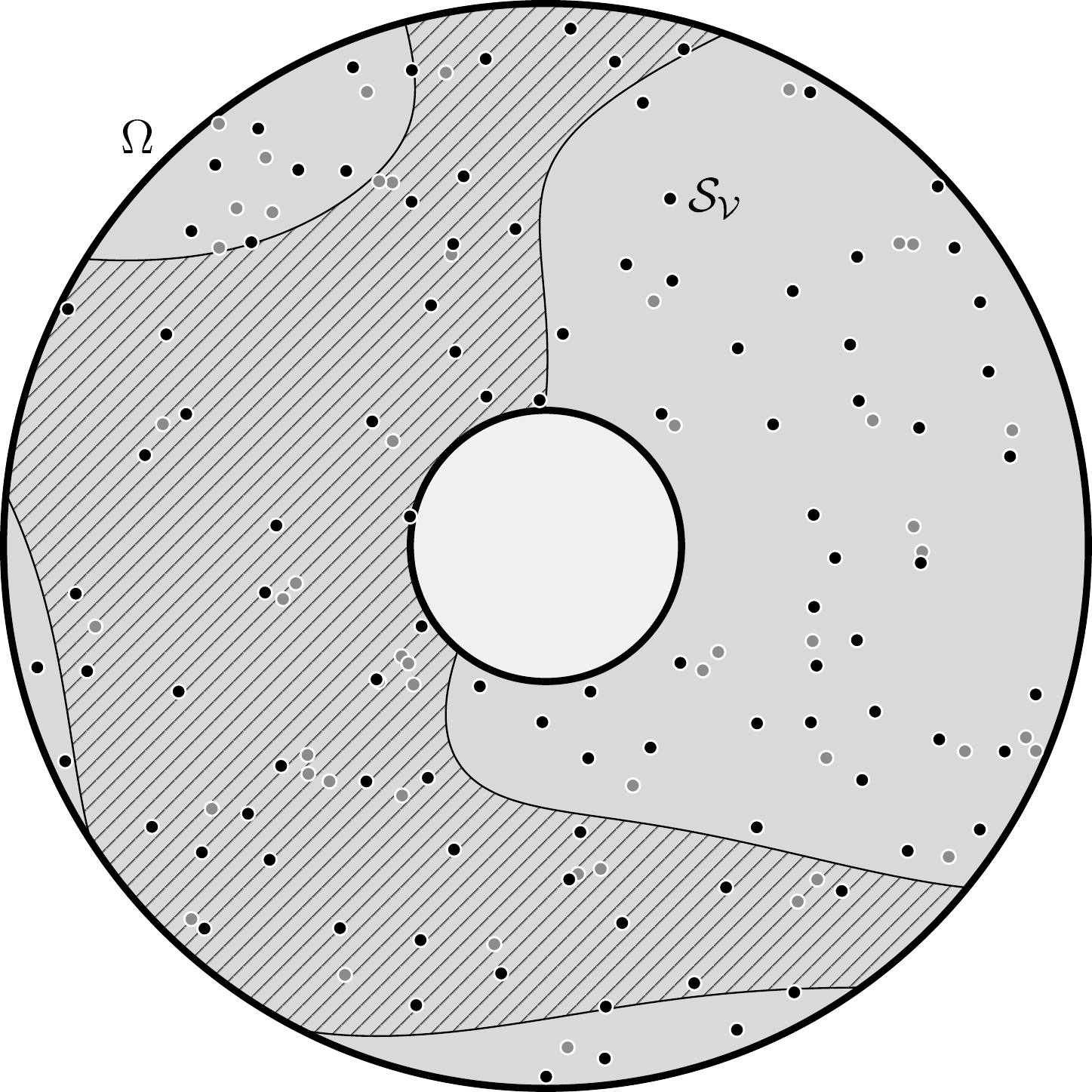}}
        \caption{Construction of the set $\mathcal S_{\mathcal V}$ from \Cref{prop:norming}. The hatched region represents the sublevel set $\{x \in \Omega: ||F(x)|^2-|G(x)|^2| \leq e^{- \mu n} \| |F|^2 - |G|^2 \|_{L^\infty(\hat \Omega)}\}$. Gray points represent the elements of $\mathcal S$ which were removed to ensure uniform separation.}
        \label{fig:norming}
    \end{figure}

    Vapnik--Chervonenkis' inequality \cite{VapnikChervonenkis1974} states that, given any family $\mathcal G$ of subsets of $\Omega$,
    \begin{align*}
        \mathbb P \left[ \sup_{X \in \mathcal G} \frac{\left| \tfrac 1M \# \{j: x_j \in X \} - |X|/|\Omega| \right|}{\sqrt{|X|/|\Omega|}} > \varepsilon \right] \leq 8 m^{\mathcal G}(2M) \exp(-M \varepsilon^2/4),
    \end{align*}
    where $m^{\mathcal G}(M)$, the so-called \emph{growth function} of $\mathcal G$, counts the maximal number of patterns induced by restricting $\mathcal G$ to a subset of $\Omega$ of cardinality $M$, that is,
    \begin{align*}
        m^{\mathcal G}(M) = \max_{x: [M] \to \Omega} \# \{\sigma \in 2^{[M]} | \sigma = x^{-1}(X) \text{ for some } X \in \mathcal G\}.
    \end{align*}
    The Sauer--Shelah lemma bounds $m^{\mathcal G}(k)$ by $\left( \frac{ek}{V} \right)^{V}$ for $k \geq V \geq \vc(\mathcal G)$, hence
    \begin{align*}
        &\mathbb P \left[ \sup_{X \in \mathcal G} \frac{\left| \tfrac 1M \# \{j: x_j \in X \} - |X|/|\Omega| \right|}{\sqrt{|X|/|\Omega|}} > \varepsilon \right] \leq 8 \left( \frac{2eM}{V} \right)^{V} \exp(-M \varepsilon^2/4) \\
        &= 8 \exp\left(V \left( 1 + \log 2 + \log \frac{4}{\varepsilon^2} + \log \frac{\varepsilon^2}{4} \frac{M}{V} - \frac{\varepsilon^2}{4} \frac{M}{V} \right) \right).
    \end{align*}
    With $\varrho(x) := x-\ln(x)$ and $\tau := 1 + 3\log 2 - 2 \log \varepsilon$, the last expression reads $8 \exp(V(\tau - \varrho(\frac{\varepsilon^2}{4} \frac{M}{V})))$. Writing $\vartheta(x) = x + \ln(x) + \frac12$, we have $\varrho \circ \vartheta(x) > x$. To bound the probability of the event above by $\frac{1}{100}$, it thus suffices to choose $\frac{M}{V} \geq \frac{4}{\varepsilon^2} \vartheta\left( \tau + \frac1V \log 800\right)$. Applying this bound to the family $\mathcal F$ with $\varepsilon = \frac12 \sqrt{t}$ and $V = 36 N$, we find that choosing $M = \left\lfloor \frac{3000+1200|\log t|}{t} N \right\rfloor \geq \frac{16}{t} (5 + 2|\log t|) 36N$ ensures that each set in $\mathcal F$ contains at least $\frac t2 M$ points of $\mathcal S$.

    For any $\delta > 0$, the probability for any two points in $\mathcal S$ to be $\delta N^{-\frac{1}{m}}$-close is bounded by $ \delta^{m} N^{-1} |\mathbb B^{m}|/ |\Omega|$. Since there are $\frac{M(M-1)}{2}$ pairs of points in $\mathcal S$, the expected number of such pairs is bounded by $\frac12 M^2 \delta^{m} N^{-1} |\mathbb B^{m}| / |\Omega|$. Upon choosing
    \begin{align*}
        \delta = \left( \frac{tN |\Omega|/|\mathbb B^{m}|}{2M} \right)^{\frac1{m}} \geq \left( \frac{t^2 |\Omega|/|\mathbb B^{m}|}{6000+2400 |\log t|} \right)^{\frac1{m}},
    \end{align*}
    the expected number of pairs is at most $\frac{t}{4}M$, and by Markov's inequality, the probability that the number of $\delta$-close pairs does not exceed $\frac t3 M$ is at least $25\%$. Hence, there exists a set $\mathcal S$ with at most $\frac t3 M$ pairs of $\delta$-close points and at least $\frac t2 M$ points in each subset $X \in \mathcal F$. Deleting a subset of $\mathcal S$ of size $\frac t3 M$ covering the set of pairs of $\delta$-close points yields the desired set $\mathcal S_{\mathscr V}$.

    To obtain the claimed bounds on $\delta$ and $\mu$ under the assumption $|\hat \Omega \setminus \Omega| \leq 2^{-2m-2} |\hat \Omega|$, we choose $t = 2^{-2m-1}$. It is easy to check that with this choice, $t|\Omega| + |\hat \Omega \setminus \Omega| < 2^{-2m} |\hat \Omega|$, making this a permissible choice of $t$, and to verify the claimed lower bound on $\delta$.
\end{proof}

\section{Tail bounds on the reproducing kernel}\label{sec:tails}

\subsection{True polyanalytic Fock spaces in higher dimensions}

Recall the orthonormal basis of $L^2(\mathbb C;e^{-|\cdot|^2})$ given by
\begin{align*}
    \phi_{m,n}(z,\bar z)
    &=
    \frac{1}{\sqrt{m!n!}}\sum_{r=0}^{\min(m,n)}
    (-1)^r r!\binom mr\binom nr z^{m-r}\bar z^{n-r}, & &m, n \in \mathbb N.
\end{align*}
For $\mathbf n, \mathbf q \in \mathbb N^d$, set $\Phi_{\mathbf n,\mathbf q}(z,\bar z) = \prod_{i=1}^d \phi_{n_i,q_i}(z_i,\bar z_i)$. For $\mathbf q \in \mathbb N^d$, let $\mathcal F^d_{\mathbf q}$ denote the closed span of $\Phi_{\mathbf n,\mathbf q}$, $\mathbf n \in \mathbb N^d$. For $L \in \mathbb N$, denote $\mathcal F^d_L = \bigoplus_{|\mathbf q| \leq L}{\mathcal F^d_{\mathbf q}}$. This is a reproducing kernel Hilbert space with orthonormal basis $\left\{ \Phi_{\mathbf n, \mathbf q}: |\mathbf q| \leq L, \mathbf n \in \mathbb N^d \right\}$ 
and reproducing kernel
\begin{align}
    \label{eq:reproducing_kernel}
    K_{L}^d(z,w) = \sum_{|\mathbf q| \leq L} \sum_{\mathbf n \in \mathbb N^d} \Phi_{\mathbf n, \mathbf q}(z,\bar z) \overline{\Phi_{\mathbf n, \mathbf q}(w,\bar w)}.
\end{align}
The space $\mathcal F^d_L$ is $U(d)$-invariant. Indeed, $\mathcal F^d_L$ is precisely the closed subspace of $L^2(\mathbb C^d;e^{-|\cdot|^2})$ consisting of polyanalytic functions of order $L$. Since polyanalyticity and the inner product on $L^2(\mathbb C^d;e^{-|\cdot|^2})$ are preserved under holomorphic rotation, the space $\mathcal F^d_L$ is as well. The $U(d)$-invariance can be exploited to obtain convenient tail bounds on the series \eqref{eq:reproducing_kernel} defining the reproducing kernel. 

For $N \in \mathbb N$, let $P_{L,N}$ denote the orthogonal projection onto $\spanop\{\Phi_{\mathbf n,\mathbf q} : |\mathbf q| \leq L,\, |\mathbf n| \leq N\}$, and let $K_{L,N}(z,w) = \sum_{|\mathbf q| \leq L}\sum_{|\mathbf n| > N} \Phi_{\mathbf n,\mathbf q}(z,\bar z)\overline{\Phi_{\mathbf n,\mathbf q}(w,\bar w)}$ denote the reproducing kernel of its orthogonal complement in $\mathcal F^d_L$.

\begin{lemma}
    \label{lem:tail_bound}
    Let $L \in \mathbb N$ and $N\geq \max(L + 9,3 L)$. Then
    \begin{align*}
        \|K_{L,N}(z,\cdot)\|_{L^2(\mathbb C^d;e^{-|\cdot|^2})} \leq C_{L} \left( 1 - \frac{e|z|^2}{N}\right)^{-\frac12}  N^{\frac L2} \left( \frac{e |z|^2}{ N} \right)^{\frac{N-2L}2}
    \end{align*}
    holds for all $|z| \leq \sqrt{N/e}$.
\end{lemma}

\begin{proof}
    Since $\mathcal F^d_L$ is $U(d)$-invariant and the reproducing kernel is unique, $\| K_L(z,\cdot) \|_{L^2(\mathbb C^d;e^{-|\cdot|^2})} = \| K_L(|z| e_1, \cdot) \|_{L^2(\mathbb C^d;e^{-|\cdot|^2})}$. Write $x = |z|$. Since 
    \begin{align*}
        \phi_{n,m}(0,0) = \begin{cases} 0 & n \neq m, \\
                                                                 (-1)^n & n = m,
                                                    \end{cases}
    \end{align*}
    the only nonzero summands $\Phi_{\mathbf n,\mathbf q}$ in the series expression for $(1-P_{N}) K_L(z,\cdot)$ occur when $\mathbf n = (m,q_2,\ldots,q_d)$. In this case, $\Phi_{\mathbf n,\mathbf q}(x e_1) = \pm\phi_{m,q_1}(x,x)$. Hence,
    \begin{align*}
        \| K_{L,N}(x e_1,\cdot) \|^2 = \sum_{\substack{0 \leq q_1 \leq L \\ |\mathbf q'| \leq L - q_1 \\ m > N-|\mathbf q'|}} \phi_{m,q_1}(x,x)^2 \leq C_L \sum_{\substack{0 \leq q_1 \leq L \\ m > N-L}} \phi_{m,q_1}(x,x)^2.
    \end{align*}

    We begin by estimating $\phi_{m,q_1}(x,x)$ for $x \in (0,\sqrt{m})$ by applying Cauchy's estimates to the formula $\phi_{m,q_1}(z) = \frac{(-1)^{q_1}}{\sqrt{q_1!}} e^{|z|^2} \partial_z^{q_1} (e^{-|z|^2} \frac{z^m}{\sqrt{m!}})$ at $z = \bar z = x$, using a polydisk of radius $\frac{x}{m}$ around $(z,\bar z) = (x,x)$:
    \begin{align*}
        |\phi_{m,q_1}(x,x)| &= \frac{1}{\sqrt{q_1!}} e^{|x|^2} \left| \partial_z^{q_1} (e^{-|x|^2} \frac{x^m}{\sqrt{m!}}) \right| \leq \sqrt{q_1!} e^{|x|^2} \frac{m^{q_1}}{x^{q_1}} \sup_{(z,\zeta) \in \mathbb D_{(\frac xm,\frac xm)}(x,x)}\left| e^{-z \zeta} \frac{z^m}{\sqrt{m!}} \right|.
    \end{align*}
    To bound the supremum on the right-hand side, observe that $\Re(z\zeta) \geq (x-\frac{x}{m})^2$ if $m \geq 2$, and that by our assumption $x^2 \leq \frac{N}{e} \leq m$, hence $(x-\frac{x}{m})^2 \geq x^2 - 2\frac{x^2}{m} \geq x^2 - 2$. Furthermore, $|z|^m \leq (1+\frac1m)^m x^m \leq e x^m$. Thus, 
    \begin{align*}
        |\phi_{m,q_1}(x,x)| &\leq e^3 \sqrt{q_1!} \frac{m^{q_1}}{x^{q_1}} \frac{x^m}{\sqrt{m!}} \leq e^3 \sqrt{L!} \frac{m^{L}}{x^{L}} \frac{x^m}{\sqrt{m!}}.
    \end{align*}
    
    This estimate allows us to bound the tail of the reproducing kernel $K_{L,N}$, yielding
    \begin{align*}
        \| K_{L,N}(x e_1,\cdot) \|^2 &\leq C_{L} \sum_{m=N+1-L}^{\infty} m^{2L} \frac{x^{2m-2L}}{m!} \leq C_L \, x^{2L} \sum_{m=N+1-L}^{\infty} \left( \frac{ex^2}{m} \right)^{m-2L} \\
        &\leq C_L x^{2L} \left( \frac{ex^2}{N} \right)^{N-3L} \left( 1 - \frac{ex^2}{N}\right)^{-1} \leq C_L  \left( 1 - \frac{ex^2}{N}\right)^{-1} N^{L} \left( \frac{ex^2}{N} \right)^{N-2L}.
    \end{align*}
    Here, the third inequality uses that $\left(\frac{ex^2}{m}\right)^{m-2L} = \left(\frac{N}{m}\right)^{m-2L}\left(\frac{ex^2}{N}\right)^{m-2L} \leq e^{L}\left(\frac{ex^2}{N}\right)^{m-2L}$ for all $m \geq N+1-L$: for $m \geq N$ this is clear, and for $N+1-L \leq m < N$ we have $\left(\frac Nm\right)^{m-2L} \leq \left(1 + \frac{N-m}{m}\right)^{m} \leq e^{N-m} \leq e^{L}$. The claim follows by taking square roots.
    \end{proof}

    The true polyanalytic Bargmann transform $\mathcal B_q$ maps the orthonormal basis of $L^2(\mathbb R^d)$ given by the Hermite functions $(h_{\mathbf n})_{\mathbf n \in \mathbb N^d}$ to the orthonormal basis $(\Phi_{\mathbf n, \mathbf q})_{\mathbf n \in \mathbb N^d}$. By linearity, the STFT with a window $w(x) = e^{-\pi |x|^2} h(x)$ such that $h$ expands into tensor-product Hermite polynomials $H_{\mathbf q}$ as $h(x) = \sum_{|\mathbf q| \leq L} b_{\mathbf q} H_{\mathbf q}(x)$ maps $(h_{\mathbf n})_{\mathbf n \in \mathbb N^d}$ to the basis 
    \begin{align*}
        \Psi_{\mathbf n, h} := \sum_{|\mathbf q| \leq L} b_{\mathbf q} \Phi_{\mathbf n, \mathbf q}, \ \mathbf n \in \mathbb N^d.
    \end{align*}
    If $h$ is $L^2$-normalized, this basis is orthonormal, and the map $\mathcal B_h$ an isometry. Note that $\Psi_{\mathbf n,h}$ is a polynomial in $z$ and $\bar z$ of total degree at most $|\mathbf n| + L$, a fact that will be crucial in the sequel.

\section{Proof of the main theorem}\label{sec:main_proof}

In this section, we conclude the proof of \Cref{thm:main} in its restated form. The argument proceeds by first truncating the Taylor series of $F$ and $G$ to order $N$ -- denote the result by $F_N$ and $G_N$, respectively -- and later letting $N \to \infty$. With the help of \Cref{prop:norming}, we construct a uniformly $\delta$-separated sampling set $\mathcal S_{N,\lambda}$, with $\delta$ independent of $N$, such that $\||F_N|^2-|G_N|^2\|_{L^\infty(B_{\lambda \sqrt{N}}(0))}$ can be estimated in terms of $\||F_N|^2-|G_N|^2\|_{L^\infty(\mathcal S_{N,\lambda})}$ with a loss that is exponential in $N$, but independent of $\lambda$. Assuming $|F|=|G|$ on $\mathcal S_{N,\lambda}$, the tail bounds from \Cref{lem:tail_bound} yield that $\||F_N|^2-|G_N|^2\|_{L^\infty(\mathcal S_{N,\lambda})}$ is exponentially small in $N$. This truncation error improves as $\lambda$ shrinks, and it turns out that it suffices to choose an absolute, dimension independent value of $\lambda$ to close the argument.

\begin{proof}[Proof of \Cref{thm:main}]
    Let $N \in \mathbb N$, $\lambda \in (0,\frac1{\sqrt e})$ and consider $\mathcal A_{N,\lambda} = B_{\lambda \sqrt{N}}(0) \setminus B_{\frac \lambda 8 \sqrt{N}}(0)$. Denote the degree of the polynomial $h$ by $L$. Let $\mathcal V_h^N$ denote the span of $\Psi_{\mathbf n,h}$, $|\mathbf n| \leq N$, where $\Psi_{\mathbf n,h}$ are the basis elements of $\mathcal F_h$ introduced in \Cref{sec:tails}. Equivalently, $\mathcal V_h^N$ is the image of $\mathrm{span}\{h_{\mathbf n}: |\mathbf n| \leq L\}$ under the generalized Bargmann transform $\mathcal B_h$.

    The space $\mathcal V^{N}_{h}$ has complex dimension $\binom{N+d}{d}$ and consists of polynomials of degree at most $N+L$. The domain $\Omega = \mathcal A_{N,\lambda}$ satisfies $|\hat{\Omega} \setminus \Omega| = 2^{-6d} |\hat \Omega| \leq 2^{-4d-2} |\hat \Omega|$ and $|\Omega|/|\mathbb B^{2d}| \geq \frac{63}{64} \lambda^{2d}N^{d}$. By \Cref{prop:norming}, there exists a $\delta$-separated set $\mathcal S_{N,\lambda} \subseteq \mathcal A_{N,\lambda}$,
    \begin{align*}
        \delta \geq \frac{\lambda N^{\frac12}}{16} \left( \frac{63/64}{3\cdot 10^5 (1 + d)} \right)^{\frac1{2d}} \binom{N+d}{d}^{-\frac{1}{2d}}
    \end{align*}
    such that 
    \begin{align}
        \label{eq:norming_applied}
    \left\lVert |F|^2-|G|^2 \right\rVert_{L^\infty(B_{\lambda \sqrt{N}}(0))}
    \le
    e^{4 (N+L)}
    \left\lVert |F|^2-|G|^2 \right\rVert_{L^\infty(\mathcal S_{N,\lambda})}
    \end{align}
    for all $F,G \in \mathcal V^{N}_{h}$. Using the inequality $\binom{N+d}{d} \leq (N+d)^d e^d/d^d$, we find that
    \begin{align*}
        \delta &\geq \frac{\lambda N^{\frac12}}{16} \left( \frac{63/64}{3\cdot 10^5 (1 + d)} \right)^{\frac1{2d}} \left( 1 + \frac{d}{N} \right)^{-\frac12} \frac{\sqrt{d}}{N^{\frac12}\sqrt{e}} \geq \frac{4\lambda}{10^5} \sqrt{d}
    \end{align*}
    for sufficiently large $N$.
    
    In the following, we assume without loss of generality that $\|F\|_{L^2(\mathbb C^d;e^{-|\cdot|^2})}, \|G\|_{L^2(\mathbb C^d;e^{-|\cdot|^2})} \leq 1$ and write $F_N := P_N F$ and $G_N := P_N G$. By the triangle inequality,
    \begin{align*}
        \left\lVert |F_N|^2-|G_N|^2 \right\rVert_{L^\infty(\mathcal S_{N,\lambda})} &\leq \left\lVert |F|^2-|G|^2 \right\rVert_{L^\infty(\mathcal S_{N,\lambda})} \\ &+ \left\lVert |F|^2-|F_N|^2 \right\rVert_{L^\infty(\mathcal S_{N,\lambda})} + \left\lVert |G|^2-|G_N|^2 \right\rVert_{L^\infty(\mathcal S_{N,\lambda})}.
    \end{align*}
    The elementary inequality $||z|^2-|w|^2| \leq 2 \max(|z|,|w|) |z-w|$ yields
    \begin{align*}
        \left\lVert |F|^2-|F_N|^2 \right\rVert_{L^\infty(\mathcal S_{N,\lambda})} \leq 2 \max(\|F\|_{L^\infty(B_{\lambda \sqrt{N}}(0))},\|F_N\|_{L^\infty(B_{\lambda \sqrt{N}}(0))}) \| F - F_N \|_{L^\infty(B_{\lambda \sqrt{N}}(0))}.
    \end{align*}
    Combining the reproducing kernel bound $|F(z)| \leq C_{L}(1+|z|)^{\frac L2} e^{|z|^2/2} \|F\|_{L^2(\mathbb C^d;e^{-|\cdot|^2})}$ with the tail bound from \Cref{lem:tail_bound} yields
    \begin{align*}
        \left\lVert |F|^2-|F_N|^2 \right\rVert_{L^\infty(\mathcal S_{N,\lambda})} &\leq C_{\lambda,L} N^{L} e^{\lambda^2N/2} (e \lambda^2)^{N/2} \\
        &= C_{\lambda,L} N^{L} \exp((\lambda^2/2 + 1/2 + \log \lambda) N).
    \end{align*}
    Combining this estimate with \eqref{eq:norming_applied}, we find that 
    \begin{align*}
        \left\lVert |F_N|^2-|G_N|^2 \right 
        \rVert_{L^\infty(B_{\lambda \sqrt{N}}(0))}
        &\leq e^{4(N+L)} \left\lVert |F|^2-|G|^2 \right
        \rVert_{L^\infty(\mathcal S_{N,\lambda})} \\
        &+ C_{\lambda,L} N^{L} \exp\left(\left(\lambda^2/2 + 9/2 + \log \lambda\right) N\right).
    \end{align*}

    Choose any sequence $(N_j)_{j=1}^\infty$ with $N_{j+1} \geq 100 N_{j}$, ensuring that $\mathcal A_{N_{j+1},\lambda}$ and $\mathcal A_{N_{j},\lambda}$ are well-separated, and set $\lambda = \frac{1}{100}$, so that
    \begin{align*}
        \lambda^2/2 + 9/2 + \log \lambda \leq -\frac1{10} < 0.
    \end{align*}
    Let $\mathcal S = \bigcup_{j=1}^\infty \mathcal S_{N_j,\lambda}$. Then any functions $F,G \in \mathcal F^d_{\mathbf q}$ with $|F|^2 = |G|^2$ on $\mathcal S$ satisfy 
    \begin{align*}
        \left\lVert |F_{N_j}|^2-|G_{N_j}|^2 \right 
        \rVert_{L^\infty(B_{\lambda \sqrt{N_j}}(0))} &\leq C_{\lambda,L} N_j^{L} e^{- N_j/10}.
    \end{align*}
    Since $F_{N_j} \to F$ and $G_{N_j} \to G$ uniformly on compact subsets of $\mathbb C^d$, $|F|^2 - |G|^2$ vanishes identically. It is well known that continuous phase retrieval holds in the polyanalytic Fock space $\mathcal F^d_{\mathbf q}$, hence $F = \tau G$ for some $\tau \in \mathbb C$ with $|\tau| = 1$.
\end{proof}

\section{Lean formalization}

This appendix describes the formalization in Lean 4 \cite{MouraUllrich2021} of \Cref{thm:main}, building on the mathematical library Mathlib \cite{mathlib2020}.

\subsection{Lean preliminaries}

In Lean, everything has a type. Types are a primitive notion, in the spirit of sets. For example, $0$ has type $\mathbb N$, expressed as $(0:\mathbb N)$. The function $f(x) := x^2$ can have type $\mathbb R \to \mathbb R$ or $\mathbb N\to \mathbb N$, depending on the context, which in Lean must be made explicit. 
Theorems are defined as follows, with named variables and hypotheses
\begin{leancode}
theorem TheoremName -- Two dashes start a comment, (this is a comment)
    {variable1 : Type} (variable2 : Type) (variable3 : Type) : 
    -- The colon separates the hypotheses and the conclusion
    conclusion
\end{leancode}

In order to state our theorem, we will require some Lean 4 definitions. The first one is \lean{Fin d}, which denotes the set of numbers $0,\dots d-1$ as Lean's \emph{canonical} set with $d$ elements. Using this, \lean{MvPolynomial (Fin d) ℂ} denotes the set of multi-variate polynomials with variables indexed by \lean{(Fin d)} over $\mathbb C$. In other words, polynomials in $d$ variables. Below, we will use \lean{ℂ^[d]} as a shortcut notation for \lean{EuclideanSpace (Fin d) ℂ}, the $d$-dimensional complex Euclidean space.

The statement of our main theorem requires three other definitions that are not in Mathlib, so instead are defined below. The first of these is \lean{UniformlyDiscrete} which takes a real number $\epsilon$, a set $S$ on a metric space, and returns the fact that the set is uniformly discrete with $\epsilon$ separation. Given \lean{(h : MvPolynomial (Fin d) ℂ)} (i.e.~a complex polynomial in $d$ variables) we can define \lean{(PolyFockSpace h)} as the \lean{Set (X → ℂ)} of functions that can be written as $\ell^2$-combinations of functions $\Psi_{\mathbf n, h}$.  The proposition \lean{PhaseRetrievalSet} will take as input a set of functions (of type \lean{Set (X → ℂ)}) and a subset of $X$ (of type \lean{Set X}) and states that all functions can be recovered (up to a globlal phase) from their absolute values in the subset.
\subsection{The formal statement of the main theorem}

\begin{leancode}
theorem DiscretePhaseRetrieval
    (d : ℕ) (h : MvPolynomial (Fin d) ℂ) (hnonzero : h ≠ 0) :
    ∃ S : Set ℂ^[d],
      UniformlyDiscrete (4 * Real.sqrt d / 10 ^ 7) S ∧
        PhaseRetrievalSet (PolyFockSpace h) S
\end{leancode}

We next explain the definitions entering this statement.

\subsection{The complex Hermite basis}

We start with the normalized complex Hermite polynomials from
Section~\ref{sec:tails}. The polynomial $\phi_{m,n}(z,\bar z)$ is defined by
the same finite sum as in the paper.

\begin{leancode}
def HermitePoly (m n : ℕ) (z : ℂ) : ℂ :=
    (Complex.sqrt ((m)! * (n)!))⁻¹ *
      ∑ r ∈ range (1 + min m n),
      (-1) ^ r * (r)! * (choose m r) * (choose n r) *
      z ^ (m - r) * (star z) ^ (n - r)
\end{leancode}

Here \lean{(m)!} denotes $m!$, \lean{choose m r} denotes
$\binom mr$, and \lean{star z} is the complex conjugate of $z$.
The finite set \lean{range (min m n + 1)} consists of the integers
$0,\ldots,\min(m,n)$, including the upper endpoint. 
The basis in $d$ variables is obtained by taking tensor products. In the
following definitions, the dimension is an implicit parameter, declared by \lean{variable {d : ℕ}}. The function $\Psi$ is indexed by $n,q \in \N^d$. The space $\N^d$ is formally the space of functions from $\{0, \dots, d-1\}$ to $\N$, and in Lean, is written as \lean{Fin d → ℕ}.

\begin{leancode}
def Φ (n q : Fin d → ℕ) (ℂ^[d]) : ℂ :=
    ∏ i : Fin d, HermitePoly (n i) (q i) (z i)
\end{leancode}

Accordingly, \lean{Φ n q z} represents
$\Phi_{\mathbf n,\mathbf q}(z,\bar z)
=\prod_{j=1}^d\phi_{n_j,q_j}(z_j,\bar z_j)$.
The first multi-index labels the coefficients in the expansion, while the
second specifies the fixed polyanalytic level.

To work with general polynomial windows $w(x) = e^{-\pi|x|^2} w(x)$, we define the basis elements $\Psi_{\mathbf n,h} = \sum_{|\mathbf q| \leq \deg h} h_{\mathbf q} \Psi_{\mathbf n,\mathbf q}$ introduced in \Cref{sec:tails}, normalized by the $L^2$-norm of $w$. The coefficient of $x^q$ of a polynomial \lean{h} is extracted via \lean{h.coeff q}. The notation \lean{∑'} denotes infinite sums. Note that in this case, the support of \lean{q} is finite, so the limit of  \lean{∑'} is trivial.

\begin{leancode}
def Ψ (h : MvPolynomial (Fin d) ℂ) (n : Fin d → ℕ) (z : ℂ^[d]) : ℂ :=
    (Real.sqrt (∑' q, ‖h.coeff q‖ ^ 2))⁻¹ * ∑' q, h.coeff q * Φ n q z
\end{leancode}

Now we can define the \lean{PolyFockSpace h} as the set of functions in \lean{ℂ^[d]→ℂ} which are $\ell^2$-weighted linear combinations of the functions \lean{Ψ h}. The functions \lean{Ψ h} are parametrized by $\N^d$, and thus the coefficients are in \lean{ℓ²(Fin d → ℕ,ℂ)}.

\begin{leancode}
def PolyFockSpace (h: MvPolynomial (Fin d) ℂ) : (Set (ℂ^[d]→ℂ)) := 
    {f | ∃ w : ℓ²(Fin d → ℕ,ℂ), f = fun z ↦ ∑' n,  w n * Ψ h n z}
\end{leancode}

\subsection{Uniform separation and phase retrieval}

Uniform discreteness in Lean can be defined for an arbitrary type \lean{X} which is known to be a metric space \lean{[MetricSpace X]}. This endows \lean{X} with a distance function \lean{dist : X → X → ℝ}. The statement \lean{UniformlyDiscrete} is then a proposition that takes a $\epsilon$ and a subset of $X$ as input, and returns the fact that for all $x\neq y$, the distance between $x$ and $y$ is greater than zero. The definition \lean{UniformlyDiscrete} does not require $\epsilon>0$ and in particular may be a vacuous statement, but the main theorem \lean{DiscretePhaseRetrieval} only considers the case $\epsilon = 4\sqrt{d}/10^7>0$.

\begin{leancode}
def UniformlyDiscrete 
    {X:Type*} [MetricSpace X] 
    (ε : ℝ) (S : Set (X)) : Prop := 
      ∀ x ∈ S, ∀ y ∈ S, x ≠ y → ε ≤  dist x y
\end{leancode}

The phase-retrieval property is defined with a similar level of generality, for an arbitrary class of
complex-valued functions on a type \lean{X}.

\begin{leancode}
def PhaseRetrievalSet {X : Type*} (F : Set (X → ℂ))
    (S : Set X) : Prop :=
      ∀ f ∈ F, ∀ g ∈ F, (∀ s ∈ S, ‖f s‖ = ‖g s‖) → ∃ θ : ℂ, ‖θ‖ = 1 ∧ f = θ • g
\end{leancode}

For a complex number, Lean's norm \lean{‖f s‖} is the absolute value
$|f(s)|$. The hypothesis therefore expresses equality of the sampled
magnitudes on $S$. In the conclusion, \lean{•} denotes scalar multiplication
of functions, so \lean{f = θ • g} means $f(x)=\theta g(x)$ for every $x$.
The scalar $\theta$ has modulus one and is independent of $x$.
Taking \lean{X} to be \lean{ℂ^[d]} and the function class to be
\lean{PolyFockSpace q} gives precisely the phase-retrieval property in the
statement above. The STFT formulation in the introduction follows from
\eqref{eq:intro_bargmann}; the displayed Lean theorem is stated directly for
the Fock space.

\subsection{Interpretation in STFT language}

The formulation of \Cref{thm:main} (STFT-version), technically a corollary of \Cref{thm:main}, was formalized in Lean as well. The formulation is structurally similar to the formalization of \Cref{thm:main} (Polyanalytic-version), switching the space of polyanalytic functions for the range of the STFT with window given by $h$.

\begin{leancode}
theorem DiscretePRSTFT
    (d : ℕ) (h : ℂ[Fin d]) (hnonzero : h ≠ 0) :
      ∃ S : Set ℂ^[d],
        UniformlyDiscrete (4 * Real.sqrt d / 10 ^ 7) S ∧
        PhaseRetrievalSet (Set.range (STFT (window h))) S 
\end{leancode}

 In order to state the theorem, one must define the STFT in Lean as a map $L^2(\mathbb R^d) \to L^2(\mathbb R^d) \to (\mathbb C^d \to \mathbb C)$ (i.e.~a map that takes a window and then returns a map from $L^2$ to the STFT, encoding the space/frequency as a real/imaginary part). To define the STFT, we first define time-frequency shifts in $L^2(\mathbb R^d)$. This is defined using \lean{DomAddAct.mk}, which turns $\mathbb R^d$ into a left-acting group in $L^2$, via the action \lean{(DomAddAct.mk x +ᵥ f)(y)  = f (y-x)}. Using this, as well as the Fourier transform \lean{F}, we define

\begin{leancode}
def timeFreqShift {d : ℕ} (x ξ : ℝ^[d]) (f : L²(ℝ^[d],ℂ)) : L²(ℝ^[d],ℂ):=
    F ((DomAddAct.mk ξ) +ᵥ (F⁻ (DomAddAct.mk -x) +ᵥ f))
\end{leancode}

This allows us to define the STFT directly, as a dot product.

\begin{leancode}
def STFT {d : ℕ} (g f : L²(ℝ^[d],ℂ)) : ℂ^[d] → ℂ :=
    fun z ↦ ⟪timeFreqShift z.re z.im g, f⟫_ℂ
\end{leancode}

Given a polynomial \lean{h}, one can define the window $e^{-\|x\|^2} h(x)$ directly. However, one wants to consider it as an element of $L^2$. The Lean function \lean{MemLp.toLp} allows one to interpret a function as an element of $L^p$, as long as one provides a proof of integrability. For the purpose of explaining the  statement, the  proof of integrability (which does not change the meaning of the statement) is omitted with the placeholder \lean{sorry} in the showcase file, which records only the statements and not the proofs.

\begin{leancode}
def window (h : ℂ[Fin d]) :  L²(ℝ^[d],ℂ):=
    MemLp.toLp 
      (fun x ↦ Real.exp (-‖x‖^2) * (h.eval (x:ℂ^[d]))) 
      (by sorry)
\end{leancode}

\subsection{Source code}

The file \lean{Showcase.lean} contains the definitions and theorem statement
presented above, with the proof of \lean{PolyDiscreteSPR} replaced by the
placeholder \lean{sorry}. This separates the task of reading the formalized
statement from that of inspecting its proof. The header of the file identifies
\lean{DiscretePR.PolyDiscreteSPR_proved} in \lean{Main.lean} as the proved
theorem, with the definitions repeated in \lean{Definitions.lean} inside the
namespace \lean{DiscretePR}. The showcase file itself records the statement
and does not supply its proof. The Lean code is available at the accompanying GitHub repository:
\begin{center}   \url{https://github.com/josefgreilhuber/DiscretePhaseRetrieval}.
\end{center}

\section*{Usage of Large Language Models}

The problem that we study here has had an important influence on the authors. In particular, it was the main objective of the PhD thesis of the third author. 

Discussions between the third and fourth authors on this problem have been ongoing for more than a year and discussions between the second and fourth authors began during Fall 2025, while the second author was visiting ETH Zürich. During that
time, a plausible strategy was developed, but executing the technical details remained unclear.

During March 2026, the first and last authors had success Lean verifying two related research papers \cite{abdalla2026stable,bertolini20262} and using GPT-5.4 as a research partner.
At this point, these two authors met at NYU and tried to get GPT-5.4 to execute the technical details following a proof sketch devised by the authors.
Despite using multiple page prompts and relatively precise instructions, we could not get GPT-5.4 to converge to a meaningful solution following the sketch.
However, discussions between the authors continued, and after roughly two months we were able to execute a proof using GPT-5.5 in the case $d=1$ in the Fock space. It was clear from the proof strategy that the argument would generalize to Hermite windows in all dimensions. This was implemented shortly after,
with the help of GPT-5.5. 

The main insight provided by GPT-5.5 that was not known to the authors was the precise VC dimension bound in \Cref{lem:VCdimension}. Notably, VC dimension
arguments had been used in previous work \cite{abdalla2025sharp} of the fourth author, but the actual execution of  \Cref{lem:VCdimension} was not human-made. After obtaining a first proof, the authors spent a significant amount of time simplifying and refining the argument. This resulted in the more general, shorter, and more precise proof presented here.

The Lean verification in this project was also very interesting. The original proof suggested by GPT-5.5 required the formalization of several deep facts in algebraic geometry. The first edition of the formalization was done by the first author and resulted in 155k lines of code to formalize the $d=1$ case in the Fock space. A new ``cleanup" workflow was then designed by the third and fourth authors which, remarkably, reduced the code to around 22k lines (this was executed with Fable 5). Generalizing the result to Hermite windows then roughly doubled the lines of code.

Finally, the second author carefully analyzed the proof and the Lean code (the latter with extensive use of Fable 5.1), making it substantially shorter and clearer as well as applicable to all windows in $e^{-\pi |x|^2} \mathbb C[x_1,\ldots,x_d]$.

We remark that the manuscript is entirely human written and does not contain LLM generated text.

\section*{Acknowledgments}

L.L.~is grateful to the Azrieli Foundation for the award of an Azrieli Fellowship and acknowledges the support of this research by ISF Grant No.~854/25. J.G. was partially supported by NSF grant DMS-2247185 awarded to Eugenia Malinnikova and by the David and Lucile Packard Foundation through the Packard Science and Engineering Fellowship awarded to Alexandr Logunov. J.G. and J.D. would like to express their gratitude to ETH Z\"urich for hosting them in the fall of 2025 and winter of 2026, respectively.

\printbibliography

\end{document}